\documentclass[12pt]{article} 
\usepackage{amsmath}
\usepackage{amssymb}
\usepackage{amsthm}
\usepackage[top=1in, bottom=1in, left=1in, right=1in]{geometry}
\usepackage{xcolor}
\usepackage{hyperref}

\newtheorem{theorem}{Theorem}[section]
\newtheorem{lemma}[theorem]{Lemma}
\newtheorem{prop}[theorem]{Proposition}
\newtheorem{coro}[theorem]{Corollary}

\newtheorem{remark}[theorem]{Remark}

\newcommand{\ddbar}{\partial\bar\partial}

\renewcommand{\o}{\omega}

\title{Optimal asymptotic of the $J$ functional with respect to the $d_1$ metric}
\author{Tam\'as Darvas, Erin George and Kevin Smith}
\date{}
\begin{document}
\maketitle
\begin{abstract} We obtain sharp inequalities between the large scale asymptotic of the $J$ functional with respect to the $d_1$ metric on the space of K\"ahler metrics. Applications regarding the initial value problem for geodesic rays are presented.
\end{abstract}

\section{Introduction and main results}

 Let $(X,\omega)$ be a K\"ahler manifold. We consider the space of K\"ahler metrics $\tilde \omega$ that are cohomologous to $\omega$:
$$\mathcal H := \{\tilde \omega \textup{ K\"ahler on } X \textup{ and } [\tilde \omega]_{dR} = [\omega]_{dR}\}$$

By the $\ddbar$-lemma of Hodge theory, for all $\tilde \omega \in \mathcal H$ there exits $u \in C^\infty(X)$, unique up to a constant, such that $\tilde \omega = \omega_u:= \omega + i\ddbar u.$ Consequently,  instead of looking at $\mathcal H$ directly, it makes sense to work with the space of \emph{K\"ahler potentials} instead:
$$\mathcal H_\omega :=\{ u \in C^\infty(X) \textup{ s.t. } \omega + i\ddbar u >0 \}.$$
Clearly $\mathcal H_\omega \subset \textup{PSH}(X,\omega)$, hence all K\"ahler potentials are $\o$-plurisubharmonic ($\omega$-psh). For a comprehensive treatment of $\omega$-psh functions we refer to the recent book \cite{GZ16}.

For a quick introduction to the basics of pluripotential theory in the context of K\"ahler geometry we refer to \cite[Appendix A.1]{Da17}, \cite[Section 2]{Bl13} or \cite[Chapter 8]{GZ16}.

A main direction of research is to find K\"ahler structures such that $\mathcal H_\omega$ admits constant scalar curvature K\"ahler (csck) metrics. Such metrics are minimizers of Mabuchi's K-energy functional  $\mathcal K: \mathcal H_\omega \to \Bbb R$:
$$\mathcal K(u):=\frac{1}{V} \int_X [\log\Big(\frac{\o_u^n}{\o^n}\Big)\o_u^n - {u}\sum_{j=0}^{n-1}\text{Ric }\o\wedge\o_u^j \wedge \o^{n-j-1}] + \bar S I(u),$$
where $V = \int_X \omega^n$ is the total volume, $\text{Ric }\o$ is the Ricci curvature of the background metric $\omega$, and $\bar S = \frac{1}{V}\int_X S_\omega \omega^n$ is the average scalar curvature of $\o$, which is also an invariant of the class $\mathcal H_\o$ (see \cite[eq. (4.40)]{Da17}). Lastly, $I:\mathcal H_\o \to \Bbb R$ is the \emph{Monge--Amp\`ere energy} (sometimes called Aubin-Yau energy), one of the most basic functionals of K\"ahler geometry:
$$I(u):= \frac{1}{(n+1)V}\sum_{j=0}^n \int_X u \o^j \wedge \o_{u}^{n-j}.$$
For a more detailed analysis of $I$ we refer to \cite[page 111]{Bl13} and \cite[Section 3.7]{Da19}. Closely related to $I$, the $J$ functional $J: \mathcal H_\omega \to \Bbb R$ is defined as follows:
$$J(u) = \frac{1}{V}\int_X u \o^n - I(u).$$
Using Stokes theorem again, it can be showed that $J(u) \geq 0$, and in many ways $J$ acts as a norm-like expression on $\mathcal H_\o$. This aspect will be featured prominently in this work.

By definition, the space of K\"ahler potentials $\mathcal H_\o$ is a convex open subset of $C^\infty(X)$, hence one can think of it as a trivial ``Fr\'echet manifold". As such, one can introduce on $\mathcal H_\o$ an $L^1$ type Finsler metric with relevant underlying geometry \cite{Da15}. If $u \in \mathcal H_\o$ and $\xi \in T_u \mathcal H_\o \simeq C^\infty(X)$, then the $L^1$-length of $\xi$ is given by the following expression:
\begin{equation}\label{eq: Lp_metric_def}
\| \xi\|_{u} = \frac{1}{V}\int_X |\xi| \o_u^n.
\end{equation}
The corresponding $L^2$ type metric recovers the Riemannian geometry of Mabuchi \cite{Ma87} (independently discovered by Semmes \cite{Se92} and Donaldson \cite{Do99}, studied later by Chen \cite{Ch00}). For more details we refer to \cite[Chapter 3]{Da17}.

To the Finsler metric in \eqref{eq: Lp_metric_def} one associates a path length pseudo-metric $d_1(\cdot,\cdot)$. As proved in \cite{Da15}, $d_1$ is actually a metric and $(\mathcal H_\omega,d_1)$ is a geodesic metric space, whose abstract completion can be identified with $(
\mathcal E^,d_1)$, where $\mathcal E^1 \subset \textup{PSH}(X,\omega)$ is a space of potentials introduced by Guedj--Zeriahi \cite{GZ07}, with connections to earlier work of Cegrell.

Let us assume that momentarily that $(X,J)$ does not admit global holomorphic vectorfieds.  Tian conjectured that existence of csck metrics in $\mathcal H_\omega$ is equivalent to $J$-properness of the K-energy functional \cite{Ti97,Ti00}:
\begin{equation}\label{eq: Tian_est}
\mathcal K(u) \geq C_J J(u) + D_J, \ \ u \in \mathcal H_\omega \cap I^{-1}(0),
\end{equation}  
where $C_J,D_J$ are some positive contants dependent only on $(X,\omega)$. Necessity of \eqref{eq: Tian_est} was pointed out in \cite{BDL2}, building on techniques of \cite{DR15}. Chen and Cheng proved that existence of csck metrics in $\mathcal H_\omega$ is equivalent to
\begin{equation}\label{eq: C_est}
\mathcal K(u) \geq C_d d_1(0,u) + D_d, \ \ u \in \mathcal H_\omega \cap I^{-1}(0), 
\end{equation}
where $C_d,D_d$ are some positive contants dependent on $(X,\omega)$ \cite{CC2}. As pointed out in \cite{Da15, DR15}, this last estimate is equivalent to  \eqref{eq: Tian_est}, since there exists  constants $m,M,D>0$ such that
\begin{equation}\label{eq: d_J_comparison}
m J(u) - D \leq d_1(0,u) \leq M J(u) + D, u \in \mathcal H_\omega \cap I^{-1}(0).
\end{equation}
By the analysis of \cite{Da15} one can choose $M =2$ and $m = 2^{-2n-6}$ in the above inequality (see Proposition \ref{prop: d_1_growth_J} below). Since the $C_J$ and $C_{d}$ are linked to the uniform version of K-stability \cite{BBJ}, it interesting to know what the optimal conversion rate is between these two constants. For this the optimal values of $m$ and $M$ need to be found in \eqref{eq: d_J_comparison}, and this is what we investigate in this paper, first on toric K\"ahler manifolds $(X,\omega)$:

\begin{theorem}\label{thm: main_thm} Let $(X,\omega)$ be a toric K\"ahler manifold. Then there exists a constant $D>0$ such that 
\begin{equation}\label{eq: main_thm}
\frac{2}{n+1}\cdot{\left(\frac{n}{n+1}\right)}^n J(u)  - D \leq d_1(0,u)\leq 2 J(u) + D, \ \ u \in \mathcal H^T_\omega \cap I^{-1}(0),
\end{equation}
and the constants multiplying $J(u)$ are sharp.
\end{theorem}

By the above result, at least in the toric case, the optimal constant $m$ in \eqref{eq: d_J_comparison} has linear decay with respect to $\dim X$, and not exponential, as previously thought. Surprisingly, neither of the optimal constants depend on the choice of K\"ahler metric $\omega$.

Similar flavour results were recently obtained by Sjöstr\"om Dyrefelt, who proved optimal inequalities between the asymptotics of other ``$J$-type" functionals, when the so called $J$-equation has no solution \cite[Theorem 1.2]{SZ19}. As pointed out in this latter work, inequalities like the ones in \eqref{eq: main_thm} often give criteria existence for canonical metrics, something we hope to investigate in the future.

Using Legendre transforms the inequality in the above result can be transformed into a pair of sharp integral inequalities involving convex functions defined on a convex domain, that we now present:

 \begin{theorem}\label{thm: main_thm_convex} Let $P \subset \Bbb R^n$ be a bounded open convex set. Then for any convex $\phi \in L^1(P)$, satisfying $\int_P \phi= 0$, the following sharp inequality holds:  
\begin{equation}\label{eq: main_thm_convex} 
 -\frac{2}{n+1}\cdot{\left(\frac{n}{n+1}\right)}^n\inf_P \phi \leq \frac{1}{\mu(P)}\int_P |\phi| d\mu\leq -2\inf_P \phi,
 \end{equation}
where  all integrations are in terms of the Lebesque measure.
\end{theorem}

We prove sharpness of \eqref{eq: main_thm_convex} by providing concrete extremizing potentials, and it would be interesting to characterize all such extremizers. Another interesting question would be to find the optimal constants in \eqref{eq: main_thm_convex} for any fixed $P$. As these questions have potential applications in K\"ahler geometry, we hope to return to them in the future.

Despite its basic nature, we could not find the above inequality in the written convex geometry literature, even after consulting with experts \cite{Mi20}. To be clear, the emphasis here is not on novelty, but rather on the significance of \eqref{eq: main_thm_convex} in the context of K\"ahler geometry.

We conjecture that Theorem \ref{thm: main_thm} holds in case of general K\"ahler manifolds as well.  To provide strong evidence for this, we prove the radial version of this expected result. For this we recall some terminology from \cite{DL18} first. 

Let $\{u_t\} \in \mathcal R^1$ be a geodesic ray $[0,\infty) \ni t \to u_t \in \mathcal E^1$ of $(\mathcal E^1,d_1)$, emanating from $u_0 = 0 \in \mathcal H_\omega$, and normalized by $I(u_t) = 0, \ t \geq 0$. Since the $J$ functional is convex along geodesics, one can define its slope along geodesic rays (informally called the \emph{radial $J$ functional}):
$$J\{u_t\} := \lim_{t \to \infty} \frac{J(u_t)}{t}.$$
The $L^1$ speed of a ray $\{u_t\}_t$ is simply the quantity $d_1(0,u_1)$. We prove the following sharp inequality between the $L^1$ speed and the radial $J$ functional:

\begin{theorem}\label{thm: radial_main_ineq} Suppose that $(X,\omega)$ is a compact K\"ahler manifold and $\{u_t\}_t \in \mathcal R^1$ is an $L^1$ ray. Then the following sharp inequality holds:
\begin{equation}\label{eq: radial_main_ineq}
\frac{2}{n+1}\cdot{\left(\frac{n}{n+1}\right)}^n J\{u_t\} \leq d_1(0,u_1)\leq 2 J\{u_t\}.
\end{equation}
\end{theorem}

Since rays are constant speed, notice that the middle term in the above inequality could have been replaced by the limit $\lim_{t \to \infty} \frac{d_1(0,u_t)}{t}$. In particular, \eqref{eq: radial_main_ineq} would instantly follow from the conjectured inequality \eqref{eq: main_thm} for general K\"ahler manifolds.

Remarkably, the proof of \eqref{eq: radial_main_ineq} rests on the ideas yielding \eqref{eq: main_thm_convex}, despite the fact that it works for general (non-toric) K\"ahler manifolds. We refer to Section 5 for more details.

Finally, we give an application for Theorem 1.3 regarding the initial value problem for geodesic rays in $\mathcal H_\omega$. In case $\omega$ is real analytic, and one is given a real analytic function $v: X \to \Bbb R$, by an application of the Cauchy-Kovalevskaya theorem, there exists a smooth geodesic $[0,\varepsilon_v) \ni t \to u_t \in \mathcal H_\omega$ such that $u_0 =0$ and $\dot u_0 = v$ \cite{RZ17,RZ12}. Of course, $t \to u_t$ is real analytic too, and it is still not known if such $t \to u_t$ can be extended to a geodesic ray (in the metric sense). In some instances, this can be done, as pointed out in \cite{AT03}. However, as we confirm below, there is plenty of analytic initial data, for which this fails to happen. What is more, we give a general condition that initial data of $L^\infty$ rays $\{u_t\}_t \in \mathcal R^\infty$ (rays from $\mathcal R^1$ with bounded potenials) need to satisfy:
\begin{coro} Suppose that $\{u_t\}_t \in \mathcal R^\infty$ is a geodesic ray, with bounded potentials. Then the initial tangent vector $v:= \lim_{t \to 0} \frac{u_t}{t} \in L^\infty(X)$ satisfies the following sharp inequalities:
\begin{equation}\label{eq: init_val_ineq}
\frac{2}{n+1}\cdot{\left(\frac{n}{n+1}\right)}^n \sup_X  v \leq \int_X |v| \omega^n\leq 2 \sup_X v.
\end{equation}
\end{coro}
Since real analytic functions are dense among smooth ones \cite[Proposition 2.1]{Le19}, one can find plenty of real analytic $v$ for which either inequality in \eqref{eq: init_val_ineq} fails, in particular such $v$ can not be the initial tangent vector for a geodesic ray.

Interestingly, in the above result we can not pinpoint the time where the geodesic segment with (real analytic) initial tangent $v$ can not be continued anymore. 
Related to this, in \cite{Da20} the author has provided  non-smooth geodesic segments that can not be continued at a specific point, however the information there could not be linked to the initial tangent.

\paragraph{Acknowledgments.} The bulk of this project was carried out during the summers of 2017 and 2019, in a project funded by NSF grant DMS-1610202. The first author is currently partially supported by NSF grant DMS-1846942(CAREER).  We thank E. Milman and Y.A. Rubinstein for conversations related to the topic of this paper.

\section{The $L^1$ geometry of the space of K\"ahler potentials}

In this short section we recall basic facts about the path length metric $d_1$ associated to the Finsler metric \eqref{eq: Lp_metric_def}, with focus on the relationship with the $J$ functional. For a survey on this topic, we refer to \cite[Section 3.7]{Da19}

Let us recall the following comparison theorem for the $d_1$ metric \cite[Theorem 3.32]{Da19}:

\begin{theorem}\label{thm: Energy_Metric_Eqv} For any $u_0,u_1 \in \mathcal H_\omega$ we have
\begin{equation}\label{eq: Energy_Metric_Eqv}
d_1(u_0,u_1) \leq \frac{1}{V}\int_X |u_0 - u_1| \o_{u_0}^n + \frac{1}{V}\int_X |u_0 - u_1| \o_{u_1}^n\leq {2^{2n + 6}} d_1(u_0,u_1).
\end{equation}
\end{theorem}

We also recall the following concrete formula for the path length metric $d_1$ (\cite[Proposition 3.43]{Da19}):
\begin{equation}\label{eq: d_1_formula}
d_1(u,v) = I(u) + I(v) - 2 I(P(u,v)), \ \ u,v \in \mathcal H_\o,
\end{equation}
where $P(u,v)$ is the following ``rooftop" envelope:
$$P(u,v)= \sup\{w \in \textup{PSH}(X,\omega) \textup{ s.t. } w \leq u \textup{ and } w \leq v \}.$$
More concretely, $P(u,v)$ is the greatest $\o$-psh function that lies below $u$ and $v$. For properties of $P(u,v)$, we refer to \cite[Section 2.4]{Da19}.   

As already suggested by \eqref{eq: d_1_formula}, there is an intimate relationship between the metric $d_1$ and the $I$ functional. 
By inspection, $I(u+c)=I(u) + c$ for any $u \in \mathcal H_\o$ and $c \in \Bbb R$. This allows for the following bijection between metrics and potentials:
$$\mathcal H_\o \cap I^{-1}(0) \simeq \mathcal H.$$
What is more, the hypersurface $\mathcal H_\omega \cap I^{-1}(0)$ is totally geodesic within $\mathcal H_\o$ (see the discussion near \cite[(3.67)]{Da19}). Finally, let us recall \cite[Proposition 3.44]{Da17}, giving the best available asymptotic comparison between the $J$ functional and the $d_1$ metric from the literature:

\begin{prop}\label{prop: d_1_growth_J}
There exists $C=C(X,\o)> 1$ such that
\begin{equation}\label{eq: d_1_growth_J}
 2^{-2n - 6}J(u) -C \leq d_1(0,u) \leq 2 J(u) + C, \ \ u \in \mathcal H_\omega \cap I^{-1}(0).
\end{equation}
\end{prop}
\begin{proof} Let  $u \in \mathcal H_\omega \cap I^{-1}(0)$. By Theorem \ref{thm: Energy_Metric_Eqv} we have 
$$J(u)=\frac{1}{V} \int_X u \o^n \leq \frac{1}{V} \int_X |u| \o^n \leq 2^{2n+6}d_1(0,u),$$
implying the first estimate in \eqref{eq: d_1_growth_J}. For the second estimate, since $I(u)=0$, we have that $\sup_X u \geq 0$ and 
\begin{equation}\label{eq: somethingd_1}
d_1(0,u)=-2 I(P(0,u)).
\end{equation}
Clearly, $u - \sup_X u \leq \min(0,u)$, so 
$u - \sup_X u \leq P(0,u)$. Thus, $-\sup_X u = I(u - \sup_X u) \leq I(P(0,u)).$ 
Combined with \eqref{eq: somethingd_1}, we obtain that $
d_1(0,u)=-2I(P(0,u))\leq 2 \sup_X u.$
Finally, it is well known that  $\sup_X u \leq \frac{1}{V} \int_X u \o^n + C'$ for some $C'(X,\o) > 1$ \cite[Lemma 3.45]{Da19}, finishing the proof.
\end{proof}

\section{Analysis on toric K\"ahler manifolds}

In this short section we point out how the inequalities \eqref{eq: main_thm} and \eqref{eq: main_thm_convex} are related. Much of the material here is based on \cite[Section 6]{DG16} and \cite{ZZ08}, and we invite the reader to consult these works for a more thorough treatment. 

We say that $(X,\omega)$ is a toric K\"ahler manifold of complex dimension $n$ if one can embed $(\Bbb C^*)^n$ into $X$ such that the complement of $(\Bbb C^*)^n$ inside $X$ is Zariski closed. Additionally, we ask that the trivial action of $\Bbb T^n := (S^1)^n$ on $(\Bbb C^*)^n$ extends to $X$, and the K\"ahler form $\omega$ is $\Bbb T^n$-invariant.

Using the fact that $\omega$ is $\Bbb T^n$-invariant we get that
\begin{equation}\label{eq: omega_triv}
\omega = i\ddbar (\psi_0 \circ L) \ \textup{ on } \ (\Bbb C^*)^n,
\end{equation}
where $\psi_0 \in C^\infty(\Bbb R^n)$ and $L(z_1,z_2,\ldots,z_n) = (\log|z_1|,\log|z_2|,\ldots,\log|z_n|) \in \Bbb R^n.$ Since $\psi_0 \circ L$ is psh on $(\Bbb C^*)^n$, it follows that $\psi_0$ has to be strictly convex on $\Bbb R^n$ (see \eqref{eq: Hessian_calc} below), and we may choose $\psi_0(0)=0$.

By $\mathcal H^{T}$ we will denote the metrics $\omega' \in \mathcal H$ that are torus invariant, i.e., $(S^1)^n$ acts by isometries on $\omega'$. The corresponding space of toric potentials will be denoted by $\mathcal H^{T}_\omega$.

Given  $u \in \mathcal H^{T}_\omega$, comparing with \eqref{eq: omega_triv}, we can introduce the following potential 
$$\psi_u := \psi_0 + u \circ E,$$
where $E(x)=E(x_1,x_2,\ldots,x_n) := (e^{x_1},e^{x_2},\ldots,e^{x_n}), \ x \in \Bbb R^n.$  The point here is that $\omega_u = i\ddbar \psi_u \circ L$ on $(\Bbb C^*)^n$.

\paragraph{The Legendre transform.} Given $\omega_u \in \mathcal H_\o^T$, it follows from a result of  Atiyah--Guillemin-- (see \cite[Chapter 27]{CS08}) that  the ``moment
map" $\nabla \psi_u: \Bbb R^n \to \Bbb R^n$ is one-to-one and sends $\Bbb R^n$ to $P := \textup{Im }\nabla \psi_u$, which is a convex bounded polytope, independent of $u$, that can be described in the following manner:
$$P := \{l_j(s)\geq 0, \ 1\leq j \leq d  \} \subset \Bbb R^n,$$
where $l_j(s) = \langle s, v_j\rangle - \lambda_j$ are affine functions that determine the sides of $P$. 

Though we will not use it, by a theorem of Delzant, $P$ satisfies a number of properties (it is {simple}, i.e., there are $n$ edges meeting at each vertex; it is {rational}, i.e., the edges meeting at the vertex $p$ are rational in the sense that each edge is of the form $p + t u_i$, $0 \leq t < \infty$, where $u_i \in \Bbb Z^n$; it is {smooth}, i.e., these $u_1,\ldots,u_n$ can be chosen to be a basis of $\Bbb Z^n$. In fact, such Delzant polytopes $P$ determine toric K\"ahler structures $(X,\omega)$ uniquely (see \cite{Ab00} and \cite[Chapter 28]{CS08}).

Since $\psi_u: \Bbb R^n \to \Bbb R$ is convex, we can take the Legendre transform of $\psi_u$ and obtain another convex function $\phi_u$,  with possible values equal to $+\infty$:
$$\phi_u(s) = \psi_u^*(s):= \sup_{x \in \Bbb R^n}\big(\langle s, x\rangle - \psi_u(x) \big), \ \ s \in \Bbb R^n.$$

Since $\textup{Im }\nabla \psi_u = P$, it follows that $\phi_u(s)$ is finite if and only if $s \in P$. Also,  by the involutive property of Legendre transforms, for all $x \in \Bbb R^n$ we will have
$$\phi^*_u = \psi^{**}_u = \psi_u,$$
$$\phi_u(\nabla \psi_u(x)) = \langle x,\nabla \psi_u(x) \rangle - \psi_u(x) \ \ \textup{ and } \ \  \nabla \psi_u(x) = s \Leftrightarrow \nabla \phi_u(s) = x.$$

Summarizing, the Legendre transform $u \to \psi_u^*=\phi_u$ gives a one-to-one correspondence between elements of $\mathcal H_\omega^T$ and the class $\mathcal C(P)$:
$$\mathcal C(P) := \{f:P \to \Bbb R \textup{ is convex and } f - \phi_0 \in C^\infty(\overline{P})\},$$
where $\phi_0 = \psi_0^*$. In addition, as pointed out in \cite[Proposition 4.5]{Gu14},  there is a one-to-one correspondence between $(S^1)^n$-invariant elements of $\textup{PSH}(X,\omega)\cap L^\infty$ and convex functions $f:P \to \Bbb R$ for which $f - \phi_0$ is only bounded on $P$. In particular, $\phi_0 \in L^\infty(P)$.

\paragraph{The $L^1$ Finsler geometry of toric metrics.} We now describe the $L^1$ geometry of $\mathcal H_\omega$ restricted to $\mathcal H_u^T$, the space of potentials for $\Bbb T^n$-invariant K\"ahler metrics.
 
Let $[0,1]\ni t \to u_t \in \mathcal H_\omega^T$ be a smooth curve  connecting $u_0,u_1\in \mathcal H_\omega^T$. Taking the Legendre transform of the potentials $\psi_{u_t}$ we arrive at
\begin{equation}\label{eq: phi_u_t_def}
\phi_{u_t}(s):=\sup_{x \in \Bbb R^n} \left\{ \langle x,s \rangle -\psi_{u_t}(x) \right\}
=\langle x_t,s \rangle -\psi_{u_t}(x_t), \ s \in \Bbb R^n,
\end{equation}
where $x_t=x_t(s)$ is such that $\nabla \psi_{u_t}(x_t)=s$. Taking derivatives of
this identity with respect to $t$ yields
$$
\nabla^2\psi_{u_t} \cdot  \dot{x}_t =-\nabla \dot{\psi}_{u_t}.
$$
Taking $t$-derivative of \eqref{eq: phi_u_t_def} and using this formula we arrive at
\begin{equation}\label{eq: dot_ident}
\dot{\phi}_{u_t}(s)=-\dot{\psi}_{u_t}(x_t).
\end{equation}

To continue, we observe that
\begin{equation}\label{eq: Hessian_calc}
\omega_{u_t}= \frac{\partial^2 (\psi_{u_t} \circ L)}{\partial z_i \partial \overline{z_j}}= \frac{1}{4}
\frac{1}{z_i \overline{z_j}} \cdot \frac{\partial^2 \psi_{u_t}}{\partial x_i \partial x_j} \circ L 
\; \text{ on } \; (\Bbb C^*)^n.
\end{equation}
Thus on $(\Bbb C^*)^n$ we have
\begin{flalign}\label{eq: CMAE_id}
\omega_{u_t}^n & = (i\ddbar \psi_{u_t} \circ L)^n
=n! \det \left( \frac{\partial^2 (\psi_{u_t} \circ L)}{\partial z_i \partial \overline{z_j}} \right)  i^{n} (d z_1 \wedge d \bar z_1) \wedge \ldots  \wedge (d z_n \wedge d \bar z_n) 
= \nonumber \\
&=\frac{n!}{2^n }\frac{1}{\Pi_j |z_j|^2} \cdot \big( MA_{\Bbb R}(\psi_{u_t}) \circ L \big)  (d x_1 \wedge d y_1) \wedge \ldots  \wedge (d x_n \wedge d y_n),
\end{flalign}
where $MA_{\Bbb R}(\psi_{u_t}) = \det \left( \frac{\partial^2 \psi_{u_t} }{\partial x_i \partial {x_j}} \right)$ denotes the real Monge-Amp\`ere measure of the (smooth) convex function $\psi_{u_t}$. 
As a result, after using polar coordinates in each $\Bbb C^*$ component ($dx \wedge dy = r dr \wedge d \theta$), we conclude that
$$
\int_{X} |\dot u_t| \omega_{u_t}^n = \int_{(\Bbb C^*)^n} |\dot \psi_{u_t} \circ L| (i\ddbar \psi_{u_t} \circ L)^n={\pi}^n n!  \int_{\Bbb R^n} |\dot \psi_{u_t}|^p MA_{\Bbb R}(\psi_{u_t}).
$$

By \eqref{eq: dot_ident}, $\dot{\psi}_{u_t}(x)=-\dot{\phi}_{u_t}(\nabla \psi_{u_t}(x))$. Moreover , $MA_{\Bbb R}(\psi_{u_t})=(\nabla \psi_{u_t})^* d\mu(s)$, where $\mu$ is the Lebesgue measure. Therefore, a change of variables $s = \nabla \psi_{u_t}(x)$ yields
\begin{equation}\label{eq: L^1_isom}
\int_{X} |\dot u_t| \omega_{u_t}^n = \pi^n n! \int_{\Bbb R^n} |\dot \psi_{u_t}| MA_{\Bbb R}(\psi_{u_t})= \pi^n n! \int_P |\dot{\phi}_{u_t}(s)| d \mu(s),
\end{equation}
Hence the Legendre transform sends the $L^1$ geometry of $\mathcal H_\omega^T$ to the flat $L^1$ geometry of convex functions on $P$. In the particular case when $u_t := t$, we obtain the following useful formula about volumes:
\begin{equation}\label{eq: volume_formula}
\int_X \omega^n = V = \pi^n n! \int_P d\mu.
\end{equation}

Regarding the underlying path length metrics,  \eqref{eq: L^1_isom} has the following important consequence:
\begin{theorem}\label{thm: d_1_toric}Suppose $u_0,u_1 \in \mathcal H_\omega^{T}$. Then
\begin{equation}\label{eq: d_1_toric}
d_1(u_0,u_1)=\frac{1}{ \mu(P)} \int_P |\phi_{u_0}(s)-\phi_{u_1}(s)| d \mu(s),
\end{equation}
where $\mu(P)$ is the Lebesgue measure of $P$.
\end{theorem}
\begin{proof}By definition $d_1(u_0,u_1)$ is the infimum of the $L^1$ Mabuchi length of smooth curves $t \to u_t$ connecting $u_0$ and $u_1$:
$$d_1(u_0,u_1) = \inf_{t \to u_t}\frac{1}{V}\int_0^1 \int_X |\dot u_t| \omega_{u_t}^n.$$

 Similarly, the integral $\int_P |\phi_{u_0}-\phi_{u_1}| d\mu$ is equal to the infimum of flat $L^1$ length of smooth curves $t \to \phi_t$ of convex functions on $P$, connecting $\phi_{u_0}$ and $\phi_{u_1}$:
 $$\int_P |\phi_{u_0}-\phi_{u_1}| d\mu = \inf_{t \to \phi_t} \int_0^1 \int_{P}|\dot \phi_t| d\mu.$$

Comparing with \eqref{eq: L^1_isom} and \eqref{eq: volume_formula}, the identity \eqref{eq: d_1_toric} follows.
\end{proof}

\paragraph{The $I$ and $J$ functionals of toric metrics.} In this paragraph we analyze the $I$ and the $J$ functionals in terms of the Legendre transform.

As it turns out, the Monge--Amp\`ere energy is essentially the Lebesgue integral, after applying the Legendre transform. Indeed, let $[0,1] \ni t \to v_t \in \mathcal H_\omega^T$ be any smooth curve connecting $v_0 =0$ and $v_1 = u$. We then obtain the following formula:
\begin{flalign}\label{eq: Aubin_Yau_toric}
I(u) &= I(u)-I(0) = \int_0^1\frac{d}{dt}I(v_t)dt= \int_0^1\frac{1}{V}\int_X \dot v_t \omega_{v_t}^n dt = \int_0^1\frac{1}{V}\int_{(\Bbb C^*)^n} \dot v_t \omega_{v_t}^n dt. \nonumber \\
&=\int_0^1\int_{(\Bbb C^*)^n} \frac{1}{V} \dot \psi_{v_t} \circ L  (i\ddbar \psi_{v_t} \circ L)^ndt \nonumber \\
&=\frac{1}{\mu(P)}  \int_0^1 \int_{\Bbb R^n} \dot \psi_{v_t} MA_{\Bbb R}(\psi_{v_t})dt \nonumber\\
&=\frac{-1}{\mu(P)}  \int_0^1 \int_{P} \dot \phi_{v_t}(s) d \mu(s) dt=\frac{-1}{\mu(P)} \int_P (\phi_u(s) - \phi_0(s))d \mu(s),
\end{flalign}

where in the second line we used \eqref{eq: CMAE_id}, in the third line we used again the change of variables $s = \nabla \psi_{v_t}$, similar to \eqref{eq: L^1_isom}, and in the last line we used \eqref{eq: dot_ident}. Since $\phi_0 \in L^\infty(P)$ (\cite[Proposition 4.5]{Gu14}), we conclude that there exists $C = C(X,\omega)>0$ such that 
\begin{equation}\label{eq: I_magn_est}
 \frac{-1}{\mu(P)} \int_P \phi_u d\mu - C \leq I(u) \leq \frac{-1}{\mu(P)} \int_P \phi_u d\mu + C, \ \ u \in \mathcal H_\omega^T.
\end{equation}

A closed formula for the $J$ energy is likely not available in terms of the Legendre transform, but we can express its magnitude in relatively simple terms, which is sufficient for our later analysis:
\begin{prop}\label{prop: J_toric} There exists $C:=C(X,\omega)>0$ such that for all $u \in \mathcal H_\o^T$ with $I(u)=0$ we have
$$-\inf_P \phi_u - C \leq J(u) = \frac{1}{V}\int_X u \omega^n \leq -\inf_P \phi_u + C.$$
\end{prop}
The argument is adapted from \cite[Lemma 2.2]{ZZ08}.

\begin{proof} Notice that $-\inf_P \phi_u = \psi_u(0),$ hence it is enough to prove existence of $C=C(X,\omega)>0$ such that:
$$\psi_u(0) - C \leq \frac{1}{V}\int_{(\Bbb C^*)^n} u  \omega^n \leq \psi_u(0) + C.$$
Since $\psi_u(0) =\psi_0(0) + u(E(0))=u(E(0)) \leq \sup_X u$,  by \cite[Lemma 3.45]{Da19}, the first estimate is trivial. 

Next we show that there exists $C:=C(X,\omega)>0$ such that for any $v \in \textup{PSH}(X,\omega)$ and constant $h > 0$ we have
\begin{equation}\label{eq: h_est}
\int_{\{v < \sup_X v - h\}}\omega^n \leq \frac{C}{h}.
\end{equation}
Indeed, from \cite[Lemma 3.45]{Da19} we have that
\begin{flalign*}
V(\sup_X v - C(X,\omega)) \leq \int_X v \omega^n &= \int_{\{v < \sup_X v -h \}} v \omega^n + \int_{\{v \geq  \sup_X v -h \}} v \omega^n\\
& \leq (\sup_X v - h) \int_{\{ v <  \sup_X v -h\}}\omega^n + \sup_X v \int_{\{v \geq  \sup_X v -h \}} \omega^n.
\end{flalign*}
This implies \eqref{eq: h_est}. Next, since  
$\textup{Im }\nabla \psi_{u} = \textup{Im }\nabla \psi_{0}=P$ and $\nabla u \circ E = \nabla \psi_u - \nabla \psi_0$, it follows that $\nabla u \circ E$ is uniformly bounded on the unit ball $B(0,1) \subset \Bbb R^n$. This implies that
\begin{equation}\label{eq: u_est}
|u(E(x)) - u(E(0))| \leq C(P), \ \ x \in B(0,1).
\end{equation}
By \eqref{eq: h_est}, there exists a constant $h := h(X,\omega)$ such that $\{u \circ E \geq  \sup_X u - h \}$ intersects $B(0,1)$. This together with \eqref{eq: u_est}, gives the desired inequality: 
$\psi_u(0)=u(E(0))\geq \sup_X u - C \geq \frac{1}{V}\int_X u \omega^n - C.$
\end{proof}

\section{A sharp inequality for convex functions}

In this section we prove the following sharp double inequality about convex functions, stated in Theorem 1.2:

 \begin{theorem}\label{thm: thm_convex} Let $P \subset \Bbb R^n$ be a bounded open convex set. Then for $f \in L^1(P)$ convex and satisfying $\int_P \phi d\mu = 0$, the following sharp inequalities hold:  
\begin{equation}\label{eq: thm_convex} 
 -\frac{2}{n+1}\cdot{\left(\frac{n}{n+1}\right)}^n\inf_P \phi \leq \frac{1}{\mu(P)}\int_P |\phi| d\mu \leq -2\inf_P \phi,
 \end{equation}
where the integration is in terms of the Lebesgue measure.
\end{theorem}

\begin{proof} First we argue the second inequality. Let 
\begin{equation}P_- := \{x \in P : \phi \leq  0\}  \ 
\  \textup{  and } \  \ P_+ := \{x \in P : \phi > 0\} 
\end{equation}
 Since $\int_P\phi=\int_{P_-}\phi+\int_{P_+}\phi=0,$ we have
\[\int_P |\phi| = -\int_{P_-} \phi + \int_{P_+} \phi = -2\int_{P_-} \phi.\]
Furthermore, $\int_{P_-} \phi \geq \mu\left(P_-\right)\inf_{P_-} \phi \geq \mu(P) \inf_P \phi$, and the second estimate follows.

We move on to the first estimate. For all $a \in \mathbb{R}$, let 
$$P_a := \{x \in P : \phi(x) \leq a\}.$$ 
Through scaling of $\phi$, we can assume without loss of generality that $\inf_P \phi = -1$.  We first claim that 
\begin{equation}\label{eq: claim_est}
\int_P |\phi| \geq -\frac{2}{n+1}\mu\left(P_-\right)\inf_P \phi.
\end{equation}
Through translation of $\phi$ and $P$, we can momentarily assume that $\phi(0) = -1 + \varepsilon$ for an arbitrarily small $\varepsilon > 0$. By convexity of $\phi$, if $-1 + \varepsilon < a < b$, we have
\[P_b \subset \frac{b+1-\varepsilon}{a+1-\varepsilon}P_a
\text{~~and~~}
\mu\left(P_b\right) \leq {\left(\frac{b+1-\varepsilon}{a+1-\varepsilon}\right)}^n\mu\left(P_a\right).\]

But translating $\phi$ and $P$ does not change the measure of any of the sets involved, so the above relation holds for all $\varepsilon > 0$ regardless of the actual value of $\phi(0)$.  Thus, 
\begin{equation} \label{eq: P_a_P_b_est}
\mu(P_b) \leq {\left(\frac{1+b}{1+a}\right)}^n\mu\left(P_a\right) \textup{ for any } -1 < a < b.
\end{equation}

Since $\int_P \phi = \int_{P_+} \phi + \int_{P_-} \phi = 0$, we have $\int_P |\phi| = 2\int_{P_-} |\phi|$, allowing us to bound $\int_{P} |\phi|$ in the following manner, proving the claim:
\begin{align*}
    \int_P |\phi| &= 2 \int_{P_-} |\phi| = 2\int_{-1}^0 \mu\left(P_x\right)dx \geq 2\int_{-1}^0 {(1+x)}^n\mu\left(P_-\right)dx = \frac{2}{n+1}\mu(P_-),
\end{align*}
where we used that $\int f d \mu = \int_0^{+\infty} \mu\{f \geq t\} dt$ for any non-negative $\mu$-measurable $f$, estimate \eqref{eq: P_a_P_b_est} for $-1 < x < 0$, and the fact that $ P_0 = P_-$. 

Let $b >0$. Next we estimate $\int_{P_+}\phi$ similarly, applying \eqref{eq: P_a_P_b_est} for $-1 < 0 < x$:
\begin{align*}
    \int_{P_+} \phi &\geq \int_0^b \mu\left(P \setminus P_x\right)dx =\int_0^b \mu\left(P\right) - \mu\left(P_x\right)dx\\
    &\geq \int_0^b \mu\left(P\right) - {(1+x)}^n\mu\left(P_-\right)dx \\
    &= b\mu(P) - \frac{{(b+1)^{n+1}}}{n+1}\mu(P_-) + \frac{1}{n+1}\mu(P_-)
\end{align*}
We then let $b = \frac{1}{n}$ and $\frac{1}{2}\int_{P}|\phi| =\int_{P_+}|\phi| = \int_{P_-}|\phi| = A\mu(P_-)$ for some $A>0$.  This gives
\[A\mu(P_-) = \int_{P_+} \phi\geq \frac{1}{n}\mu(P) - \frac{1}{n}\cdot{\left(\frac{n+1}{n}\right)}^{n}\mu(P_-) + \frac{1}{n+1}\mu(P_-),\]
implying
\[\int_P |\phi|=2A\mu(P_-) \geq \frac{2A}{nA + {\left(\frac{n+1}{n}\right)}^n - \frac{n}{n+1}}\mu(P).\]
The right-hand side is an increasing function of $A$ and by \eqref{eq: claim_est} we know $A \geq \frac{1}{n+1}$.  This means the right hand side is minimized at this value, so
\[\int_P |\phi| \geq \frac{\frac{2}{n+1}}{{\left(\frac{n+1}{n}\right)}^n}\mu(P)
= \frac{2}{(n+1)}\cdot{\left(\frac{n}{n+1}\right)}^n\mu(P).\]

\end{proof}

Finally, we address the tightness of the bounds in the previous theorem.

\begin{prop}
\label{prop:main_theo_upper_2}
For all $n > 0$ there exists an open bounded convex subset $P  \subset \mathbb{R}^n$ and convex function $\phi: P \to \mathbb{R}$ integrating to zero such that
\[\frac{1}{\mu(P)}\int_P |\phi| d\mu= -\frac{2}{n+1}\cdot{\left(\frac{n}{n+1}\right)}^n\inf_P \phi.\]
\end{prop}

\begin{proof}
Fix $n$.  Take $P \subset \mathbb{R}^n$ to be the interior of the simplex with vertex set
\[\left\{\left(0, \dots, 0\right), \left(\frac{n+1}{n}, 0, \dots, 0\right), \dots, \left(0, \dots, 0, \frac{n+1}{n}\right)\right\}\]
and $\phi: P \to \mathbb{R}$ to be the function defined by
\[\phi(x_1, \dots, x_n) = -1 + x_1 + \dots + x_n.\]
It is clear that $P$ and $\phi$ are convex and that $\inf_P \phi = -1$.  To evaluate the integrals of $\phi$ and $|\phi|$, first let $S_x$ to be the simplex with vertices
\[\left\{\left(0, \dots, 0\right), \left(x, 0, \dots, 0\right), \dots, \left(0, \dots, 0, x\right)\right\}.\]
The volume of $S_x$ is just $\frac{x^n}{n!}$.  This is the region where $\phi$ takes values less than or equal to $x-1$.  We can use the same technique from the previous theorem to show
\begin{align*}
    \int_{P_+} \phi &= \int_1^\frac{n+1}{n} \mu(P \setminus S_x)dx = \int_1^\frac{n+1}{n} \frac{1}{n!}\left({\left(\frac{n+1}{n}\right)}^n - x^n\right)dx \\
    &= \frac{1}{n}\cdot\frac{1}{n!}{\left(\frac{n+1}{n}\right)}^n - \frac{1}{(n+1)!}{\left(\frac{n+1}{n}\right)}^{n+1} + \frac{1}{(n+1)!} \\
    &= \frac{1}{(n+1)!}
\end{align*}
and
\begin{align*}
    \int_{P_-} -\phi = \int_0^1 \mu(S_{(1-x)})d x = \int_0^1 \frac{(1-x)^{n}}{n!}dx = \frac{1}{(n+1)!},
\end{align*}
where $P_+$ and $P_-$ are again the portions of the domain where $\phi$ is positive and negative, respectively.  Combined these yield $\int_P \phi = 0$ and $\int_P |\phi| = \frac{2}{(n+1)!}$.

As $\mu(P)$ is just the volume of $S_\frac{n+1}{n}$, it follows that
\[\frac{1}{\mu(P)}\int_P |\phi| = \frac{2}{n+1}\cdot{\left(\frac{n}{n+1}\right)}^n.\]
\end{proof}

\begin{prop}\label{prop:main_theo_lower}For all $n>0$ there exists  an open bounded convex set $P  \subset \mathbb{R}^n$  such that
\[\sup_{\phi}\frac{\int_P |\phi|}{-\mu(P)\inf_P \phi} = 2.\]
\end{prop}
\begin{proof}
First, suppose $n=1$ and let $P = (0,1)$.  Now, for all integers $m$, let
\[\phi_m(x) = \begin{cases}
2m-1-2m^2x \text{ if $x < \frac{1}{m}$} \\
-1 \text{ if $x \geq \frac{1}{m}$.}
\end{cases}\]
$\phi_m$ is convex, and it can be seen that $\int_P \phi_m = 0$.  It is also the case that
\[-\frac{\int_P |\phi_m|}{\mu(P)\inf_P \phi_m} = -2\int_{P_-} \phi_m = 2 - \frac{2}{m}.\]
Letting $m \to \infty$ shows that the supremum for $n=1$ is indeed 2.

For $n > 1$, let $P_n = {(0,1)}^n$ and $\phi_{n,m}(x_1, \dots, x_n) = \phi_m(x_1)$.  $\phi_{n,m}$ is still convex and we still have that $\int_{P_n} \phi_{n,m} = 0$ and $\int_{P_n} |\phi_{n,m}| = 2- \frac{2}{m}$.  Again, we can let $m \to \infty$ to show that the supremum is 2 for all $n$.
\end{proof}

\section{Proof of Theorems \ref{thm: main_thm} and \ref{thm: radial_main_ineq}}

\begin{proof}[Proof of Theorem \ref{thm: main_thm}] Using Theorem \ref{thm: d_1_toric} and Proposition \ref{prop: J_toric} the inequalities of \eqref{eq: main_thm} are equivalent to the inequalities of \eqref{eq: main_thm_convex}. Theorem \ref{thm: main_thm} is now a consequence of Theorem \ref{thm: thm_convex}
\end{proof}

In proving Theorem 1.3 we heavily rely on the formalism developed in \cite{DL18} regarding the metric space of geodesic rays. By $\mathcal R^1$ we denote the space of $L^1$ Mabuchi geodesic rays $[0,\infty) \ni t \to u_t \in \mathcal E^1$ that are normalized by $u_0 =0$ and $I(u_t) =0, \ t \geq 0$.

In \cite{DL18} the following metric was introduced for $\{u_t\}_t,\{v_t\}_t \in \mathcal R^1$:
$$d_1^c(\{u_t\}_t,\{v_t\}_t) = \lim_{t \to \infty} \frac{d_1(u_t,v_t)}{t}.$$
We know that $(\mathcal R^1,d_1^c)$ is complete \cite[Theorem 1.3 and 1.4]{DL18}, moreover the space of normalized bounded geodesic rays $\mathcal R^\infty$ is dense in $\mathcal R^1$ \cite[Theorem 1.5]{DL18}. Due to this and the next result it is  enough to prove Theorem \ref{thm: radial_main_ineq} for bounded geodesic rays:

\begin{lemma}\label{lem: R^infty_approx} Suppose that $\{u^j_t\}_t,\{u_t\}_t \in \mathcal R^1$ such that $d_1^c(\{u^j_t\},\{u_t\}_t) \to 0$. Then $J\{u_t^j\} \to J\{u_t\}$. 
\end{lemma}

\begin{proof} By the lemma below we have that $J\{u_t\} = \sup_X u_1$ and $J\{u^j_t\} = \sup_X u^j_1$. Since $d_1(u^j_1,u_1) \leq d_1^c(\{u^j_t\}_t,\{u^j_t\}_t) \to 0$ from \cite[Theorem 5]{Da15} we have that $\|u^j_1 - u_1\|_{L^1} \to 0$. Hartogs' lemma \cite[Proposition 8.4]{GZ16} now implies that $J\{u^j_t\}=\sup_X u^j_1 \to \sup_X u_1=J\{u_t\}$. 
\end{proof}

The radial $J$ functional can be expressed in very simple terms:

\begin{lemma}\label{lem: J_rad_formula} For $\{u_t\}_t \in \mathcal R^1$ we have that $J\{u_t\} = \frac{\sup_X u_l}{l}$ for any $l >0$. Moreover, in case $\{u_t\}_t \in \mathcal R^\infty$, we also have that $J\{u_t\} = \sup_X \dot u_0.$
\end{lemma}

\begin{proof} We have that $J(u_t) = \frac{1}{V} \int_X u_t \omega^n - I(u_t) = \frac{1}{V} \int_X u_t \omega^n$. By \cite[Lemma 2.2]{DL18} we obtain that 
$$J\{u_t\} = \lim_{t \to \infty} \frac{J(u_t)}{t} = \lim_{t \to \infty} \frac{\sup_X u_t}{t}.$$
Since $t \to \sup_X u_t$ is well known to be linear \cite{BBJ}, \cite[Lemma 3.2]{DX20}, the first statement follows.

In case $\{u_t\}_t \in \mathcal R^\infty$, by \cite[Theorem 1]{Da17} we also know that $\sup_X \dot u_0 = \sup_X  u_1$, proving the second statement.
\end{proof}

\begin{lemma}\label{lem: d_1_init} For $\{u_t\}_t \in \mathcal R^1$ we have that $d_1(0,u_1) = \int_X |\dot u_0|\omega^n$ and $0=I(u_1) = \int_X \dot u_0 \omega^n$.
\end{lemma}

\begin{proof} That $d_1(0,u_1) = \int_X |\dot u_0|\omega^n$  follows from \cite[Lemma 3.4]{BDL2}. The argument of \cite[Lemma 3.4]{BDL2} is seen to imply $I(u_1) = \int_X \dot u_0 \omega^n$.
\end{proof}

\begin{lemma}\label{lem: rad_legendre_sublevel} For any $\{u_t\}_t \in \mathcal R^1$ and $b \leq a \leq \sup_X \dot u_0$ we have that
$$\int_{\{\dot u_0 \geq a\}} \omega^n \leq \int_{\{\dot u_0 \geq b\}} \omega^n \leq \frac{(\sup_X \dot u_0 - b)^n}{(\sup_X \dot u_0 - a)^n} \int_{\{\dot u_0 \geq a\}} \omega^n.$$
\end{lemma}

\begin{proof} The argument uses the formalism of Legendre transforms for geodesic rays going back to \cite{RWN14}, further developed in \cite{DX20}.

The first estimate is trivial. For the second estimate we consider the Legendre transform of $\{u_t\}_t \in \mathcal R^1$:
$$\hat u_\tau := \inf_{t  \geq 0} (u_t - t\tau), \ \ \ \tau \in \Bbb R.$$
It is shown in \cite[Theorem 3.7]{DX20} that $\hat u_\tau \in \textup{PSH}(X,\omega)$ is a model potential. Moreover, by  \cite[Theorem 1]{DT20},\cite[Theorem 3.8]{DDL18} we have that 
$\int_X \omega_{\hat u_\tau}^n = \int_{\{ \hat u_\tau = 0\}} \omega^n$. 
Moreover, due to basic properties of Legendre transforms $\{\dot u_0 \geq \tau \} = \{ \hat u_{\tau} = 0\}$, in particular,
\begin{equation}\label{eq: int_est}
\int_X \omega_{\hat u_{\tau}}^n = \int_{\{ \dot u_0 \geq \tau\}} \omega^n.
\end{equation}
Let now $b \leq a \leq \sup_X \dot u_0$. Since $\tau \to \hat u_\tau$ is concave, we know that $\tau \to \big(\int_X \omega_{\hat u_\tau}^n\big)^{\frac{1}{n}}$ is concave as well (\cite[Theorem B]{DDL4} and \cite[Theorem 1.2]{WN17}). As a result, we can write that
$$\frac{(\sup_X \dot u_0 - a)}{(\sup_X \dot u_0 - b)}  \bigg(\int_X \omega_{\hat u_b}^n\bigg)^{\frac{1}{n}} +\frac{(a - b)}{(\sup_X \dot u_0 - b)}  \bigg(\int_X \omega_{\hat u_{\sup_X \dot u_0}}^n\bigg)^{\frac{1}{n}} \leq \bigg(\int_X \omega_{\hat u_a}^n\bigg)^{\frac{1}{n}}.$$
Comparing with \eqref{eq: int_est}, the result follows.
\end{proof}

Finally we arrive at the main result of this section, whose proof will be reminiscent to that of Theorem \ref{thm: main_thm_convex}:

\begin{theorem}\label{thm: radial_ineq} Suppose that $(X,\omega)$ is a compact K\"ahler manifold and $\{u_t\}_t \in \mathcal R^1$. Then the following sharp inequality holds:
\begin{equation}\label{eq: radial_opt_est}
\frac{2}{n+1}\cdot{\left(\frac{n}{n+1}\right)}^n J\{u_t\} \leq d_1(0,u_1)\leq 2 J\{u_t\}.
\end{equation}
\end{theorem}

\begin{proof} Due to Lemma \ref{lem: R^infty_approx} and \cite[Theorem 1.5]{DL18}, it is enough to prove the inequalities for $\{u_t\}_t \in \mathcal R^\infty_0$. Due to Lemma \ref{lem: J_rad_formula}, for such rays we have to argue the following estimates:
$$\frac{2}{n+1}\cdot{\left(\frac{n}{n+1}\right)}^n \sup_X \dot u_0 \leq \frac{1}{V} \int_X |\dot u_0| \omega^n \leq 2 \sup_X \dot u_0.
$$
Using re-scaling in time, we can further assume that $\sup_X \dot u_0 = 1$, hence it is enough to argue that
\begin{equation}\label{eq: radial_main_dot_ineq}
\frac{2}{n+1}\cdot{\left(\frac{n}{n+1}\right)}^n  \leq \frac{1}{V} \int_X |\dot u_0| \omega^n \leq 2.
\end{equation}
We first argue the second estimate which is much more simple.
Let $X_- := \{\dot u_0 < 0\}$ and $X_+ := \{\dot u_0 \geq  0\}$. Since $I(u_1) = \int_X \dot u_0 \omega^n$ we have that
\begin{flalign*}
\frac{1}{V}\int_X |\dot u_0| \omega^n=\frac{2}{V}\int_{X_+} \dot u_0 \omega^n \leq 2 \sup_X \dot u_0.
\end{flalign*}
To address the first estimate, we make the following preliminary calculation:
\begin{flalign}\label{eq: interm_est}
    \int_X |\dot u_0| \omega^n &= 2 \int_{X^+} \dot u_0 \omega^n = 2\int_{0}^1 \int_{\{ \dot u_0 \geq x\}} \omega^ndx \geq 2\int_{0}^1 {(1-x)}^n  \int_{\{ \dot u_0 \geq 0\}}\omega^n dx  \nonumber\\
    &= \frac{2}{n+1}\int_{\{ \dot u_0 \geq 0\}}\omega^n,
\end{flalign}
where we used that $\int f d \mu = \int_0^{+\infty} \mu\{f \geq t\} dt$ for any non-negative $\mu$-measurable $f$ and Lemma \ref{lem: rad_legendre_sublevel} for the parameters $0 \leq x \leq \sup_X \dot u_0 = 1$.

Let $b >0$. To estimate $\int_{X_-}|\dot u_0| \omega^n$, we can use a similar technique to the above:
\begin{align*}
    \int_{X_-}|\dot u_0| \omega^n &\geq \int_0^b \int_{X \setminus \{\dot u_0 \geq x \}} \omega^n dx =\int_0^b \bigg(V - \int_{ \{\dot u_0 \geq x \}} \omega^n \bigg)dx\\
    &\geq \int_0^b \bigg( V  - {(1+x)}^n\int_{ \{\dot u_0 \geq 0 \}} \omega^n \bigg)dx \\
    &= bV - \frac{{(b+1)^{n+1}}}{n+1}\int_{ \{\dot u_0 \geq 0 \}} \omega^n + \frac{1}{n+1}\int_{ \{\dot u_0 \geq 0 \}} \omega^n,
\end{align*}
where in the second line we used Lemma \ref{lem: rad_legendre_sublevel} again, for the parameters $-x \leq 0 \leq \sup_X \dot u_0 = 1$.

We now let $b = \frac{1}{n}$ and $\frac{1}{2}\int_{X}|\dot u_0| \omega^n =\int_{X_+} |\dot u_0|  \omega^n= -\int_{X_-} \dot u_0  \omega^n = A\int_{ \{\dot u_0 \geq 0 \}} \omega^n$ for some $A>0$.  This gives
\[A \int_{ \{\dot u_0 \geq 0 \}} \omega^n\geq \frac{1}{n}V - \frac{1}{n}\cdot{\left(\frac{n+1}{n}\right)}^{n}\int_{ \{\dot u_0 \geq 0 \}} \omega^n + \frac{1}{n+1}\int_{ \{\dot u_0 \geq 0 \}} \omega^n,\]
implying
\[\int_X |\dot u_0| \omega^n=2A\int_{ \{\dot u_0 \geq 0 \}} \omega^n \geq \frac{2A}{nA + {\left(\frac{n+1}{n}\right)}^n - \frac{n}{n+1}}V.\]
The right-hand side is an increasing function of $A$ and by \eqref{eq: interm_est} we know $A \geq \frac{1}{n+1}$.  This means the right hand side is minimized at this value, so
\[\int_X |\dot u_0| \omega^n \geq \frac{\frac{2}{n+1}}{{\left(\frac{n+1}{n}\right)}^n}V
= \frac{2}{(n+1)}\cdot{\left(\frac{n}{n+1}\right)}^nV.\]
This finishes the proof of \eqref{eq: radial_main_dot_ineq}.
\end{proof}

\begin{remark} The inequalities of the above theorem are sharp due to the examples produced in the toric case. Indeed, if we take $\phi(x_1,\ldots,x_n) = -1+ x_1 + \ldots + x_n$, as considered in Proposition \ref{prop:main_theo_upper_2}, then $t \to t \phi$ will give a toric ray confirming the optimality of the constant $\frac{2}{n+1}\cdot{\left(\frac{n}{n+1}\right)}^n$  in \eqref{eq: radial_opt_est}. 

Similarly, taking $\phi_{n,m}$ as in the proof of Propostion \ref{prop:main_theo_lower}, the correspondence $t \to t \phi_{n,m}$ gives a toric ray that  confirms optimality of the constant 2 in the second ineqality of \eqref{eq: radial_opt_est}. 
\end{remark}

Finally we give an application regarding the initial value problem for geodesic rays:

\begin{coro} Let $\{u_t\}_t \in \mathcal R^\infty$. Then the initial tangent vector $v:= \dot u_0 \in L^\infty(X)$ satisfies the following sharp estimates:
\begin{equation}\label{eq: init_val_ineq1}
\frac{2}{n+1}\cdot{\left(\frac{n}{n+1}\right)}^n \sup_X  v \leq \int_X |v| \omega^n\leq 2 \sup_X v.
\end{equation}
\end{coro}

\begin{proof} By Lemmas \ref{lem: J_rad_formula} and \ref{lem: d_1_init} we know that $d_1(0,u_1) = \int_X |\dot u_0| \omega^n$ and $J\{u_t\} = \sup_X \dot u_0$. \eqref{eq: init_val_ineq1} now follows directly from \eqref{eq: radial_opt_est}.
\end{proof}

\let\omegaLDthebibliography\thebibliography 
\renewcommand\thebibliography[1]{
  \omegaLDthebibliography{#1}
  \setlength{\parskip}{1pt}
  \setlength{\itemsep}{1pt plus 0.3ex}
}

\bigskip
\normalsize
\noindent{\sc University of Maryland}\\
{\tt tdarvas@umd.edu}\vspace{0.1in}\\
\noindent{\sc University of California, Los Angeles}\\
{\tt egeo@math.ucla.edu}\vspace{0.1in}\\
\noindent{\sc Columbia University}\\
{\tt kjs@math.columbia.edu}
\end{document}